\pdfoutput=1
\documentclass{amsart}
\usepackage{amssymb}
\usepackage{amsfonts}
\usepackage{extarrows}
\usepackage{tikz}
\usetikzlibrary{shapes,arrows}
\usepackage{multirow}
\usepackage{makecell}

\newtheorem{theorem}{Theorem}
\theoremstyle{plain}

\newtheorem{lemma}{Lemma}

\newtheorem{remark}{Remark}

\renewcommand\bigskip{\medskip}

\begin{document}
\title[]{On a conjecture of J. Shallit about Ap\'ery-like numbers}
\author{Zhao SHEN}
\date{\today }
\keywords{Ap\'ery-like number, combinatorial identity, regular sequence}
\markboth{}{}
\maketitle

\begin{abstract}
Put $a(n)=\sum\limits_{k=0}^{n}\binom{n}{k}\binom{n+k}{k}$, and $b(n)=v_{3}(a(n))$, for all integers $n\geqslant 0$, where $v_{3}$ is the $3$-adic valuation. In this work, we shall confirm a formula about $b(n)$, conjectured by J. Shallit in 2000. As application, we show that the sequence $(b(n))_{n\geqslant 0}$ is $3$-regular.
\end{abstract}

\subsection{Introduction}
Define $A(n)=\sum\limits_{k=0}^{n}\binom{n}{k}^{2}\binom{n+k}{k}$ and $B(n)=\sum\limits_{k=0}^{n}\binom{n}{k}^{2}\binom{n+k}{k}^{2}$, for all integers $n\geqslant 0$. These quantities, called
Ap\'ery numbers, have been introduced by R. Ap\'ery \cite{ap1} to show the  irrationality of $\zeta(2)$ and $\zeta(3)$.

For all integers $n\geqslant 0$, put $P_n(x)=\frac{1}{2^n(n!)}\frac{\mathrm{d}^n}{\mathrm{d}x^n}(x^2-1)^n$ the $n$-th Legendre polynomial,
and consider the Ap\'ery-like numbers (see for example \cite{ba,ca})
$$
a(n)=P_n(3)=\sum\limits_{k=0}^{n}\binom{n}{k}\binom{n+k}{k}.
$$
which has been used by  K. Alladi and M. L. Robinson \cite{ar} to show the irrationality of $\log 2$.
For more about Ap\'ery-like numbers, see for example \cite{co,os,su} and the references therein.

In this work, we shall show the following result, conjectured in 2000 by J. Shallit in his talk \cite[p.~18]{sh1} (see also \cite[p.~453]{as3}),
who was inspired by Advanced problem {\bf 6625} put forward in 1990 by N. Strauss and J. Shallit \cite{st1} (solved by D. Zagier \cite{st2}),
and based on computer experiments.

\begin{theorem}\label{thm1}
For all integers $n\geqslant 0$, put $b(n)=v_{3}(a(n))$. Then
\begin{equation}\label{eq1}
b(n)=\left\{\begin{array}{lcl}
&b\big(\lfloor\frac{n}{3}\rfloor\big)+ \big(\lfloor\frac{n}{3}\rfloor \bmod 2\big), &\text{if } n\equiv 0,2 \ (\bmod\, 3),\\
&b\big(\lfloor\frac{n}{9}\rfloor\big)+1, &\text{if } n\equiv 1 \ (\bmod\, 3).
\end{array}\right.
\end{equation}
where $v_{3}$ is the $3$-adic valuation, i.e., $v_{3}(0)=+\infty$, and for all integers $k\geqslant 1$, $v_{3}(k)$ is the greatest integer $d\geqslant 0$ such that $3^d$ divides $k$.
\end{theorem}

\subsection{A closed formula for $b(n)$}
In this section, we shall present a closed formula for $b(n)$ via the $3$-ary expansion of $n$.
We begin with a lemma, which is a special case of a formula obtained by E. W. Barnes \cite[p.~120]{ba} in 1908 (see also \cite[p.~268]{ra}).

\begin{lemma}\label{lem1} For all integers $n\geqslant 0$, we have
\begin{eqnarray}\label{eq6}
a(n)=\sum\limits_{k=0}^{\lfloor\frac{n}{2}\rfloor}\frac{n!}{(k!)^2(n-2k)!}2^{k}3^{n-2k}.
\end{eqnarray}
\end{lemma}
\begin{proof} Indeed, for all integers $n\geqslant 0$, E. W. Barnes has shown the following formula
$$
P_n(x)=\sum_{k=0}^{\lfloor\frac{n}{2}\rfloor}(-1)^k\frac{n!}{2^{2k}(k!)^2(n-2k)!}(1-x^2)^kx^{n-2k},
$$
from which, by putting $x=3$, we deduce at once the desired result.

Since it is difficult to obtain \cite{ba}, we give below a direct proof of the formula (\ref{eq6}).

Note that for all integers $k\geqslant 0$, we have $(1+x)^{n+k}=(1+x)^{k}(1+x)^{n}$. By comparing  the coefficient of $x^{k}$ in both sides, we obtain
\begin{eqnarray*}
\binom{n+k}{k}=\sum\limits_{i=0}^{k}\binom{k}{k-i}\binom{n}{i},
\end{eqnarray*}
from which we deduce directly
\begin{eqnarray*}
a(n)&=&\sum\limits_{k=0}^{n}\binom{n}{k}\binom{n+k}{k}=\sum\limits_{k=0}^{n}\sum\limits_{i=0}^{k}\binom{n}{k}\binom{k}{k-i}\binom{n}{i}\\
&=&\sum\limits_{i=0}^{n}\sum\limits_{k=i}^{n}\binom{n}{k}\binom{k}{k-i}\binom{n}{i}
=\sum\limits_{i=0}^{n}\binom{n}{i}\sum\limits_{k=i}^{n}\frac{n!}{(n-k)!k!}\cdot \frac{k!}{(k-i)!i!}\\
&=&\sum\limits_{i=0}^{n}\binom{n}{i}\sum\limits_{k=i}^{n}\frac{n!}{(n-i)!i!}\cdot \frac{(n-i)!}{(k-i)!(n-k)!}\\
&=&\sum\limits_{i=0}^{n}\binom{n}{i}\sum\limits_{k=i}^{n}\binom{n}{i}\binom{n-i}{k-i}
=\sum\limits_{i=0}^{n}\binom{n}{i}\binom{n}{n-i}\sum\limits_{j=0}^{n-i}\binom{n-i}{j}\\
&=&\sum\limits_{i=0}^{n}\binom{n}{i}\binom{n}{n-i}2^{n-i}.
\end{eqnarray*}
Thus $a(n)$ equals the coefficient of $x^{n}$ in $(1+x)^{n}(2+x)^{n}$.
On the other hand,
\begin{eqnarray*}
(1+x)^{n}(2+x)^{n}=(2+x^{2}+3x)^{n}=\sum\limits_{i=0}^{n}\binom{n}{i}(2+x^2)^{i}(3x)^{n-i}.
\end{eqnarray*}
The coefficient of $x^{n}$ in $(2+x^2)^{i}(3x)^{n-i}$ is $0$ if $2\nmid i$, or $\binom{i}{i/2}2^{i/2}3^{n-i}$ if $2\mid i$. So
\begin{eqnarray*}
a(n)=\sum\limits_{k=0}^{\lfloor\frac{n}{2}\rfloor}\binom{n}{2k}\binom{2k}{k}2^{k}3^{n-2k}
=\sum\limits_{k=0}^{\lfloor\frac{n}{2}\rfloor}\frac{n!}{(k!)^2(n-2k)!}2^{k}3^{n-2k}.
\end{eqnarray*}
Hence the desired result holds.
\end{proof}

\begin{theorem}\label{thm2}
For  all integers $n\geqslant 1$, write $n=\sum\limits_{j=0}^mn(j)3^j$, with $m=\lfloor \frac{\log n}{\log 3}\rfloor$, and  $n(j)\in \{0,1,2\}\ (0\leqslant j\leqslant m)$.
Let $s_{1}<s_{2}<\cdots <s_{r}$ be all the indices $j$ (if there exist) such that $n(j)=1$. Then
\begin{eqnarray}\label{eq2}
b(n)=(n \bmod 2)+\sum\limits_{i=1}^{r}(-1)^{r-i}s_{i}=\sum\limits_{i=1}^{r}(-1)^{r-i}(s_{i}+1),
\end{eqnarray}
where by convention, the sums equal $0$ if $r=0$, i.e., $n(j)\neq 1\ (0\leqslant j\leqslant m)$.
\end{theorem}
\begin{remark} If $r=0$, then $n(j)\in \{0,2\}$ for $0\leqslant j\leqslant m$, thus $n\equiv 0\equiv r\, (\bmod\, 2)$.

If $r\geqslant 1$, then we have
\begin{eqnarray*}
n=\sum_{j=0}^mn(j)3^j\equiv \sum_{i=1}^r3^{s_i}\equiv r\, (\bmod\, 2), \textrm{and}\ \sum_{i=1}^r(-1)^{r-i}= (r \bmod 2).
\end{eqnarray*}
So in all cases, we have $n\equiv r\, (\bmod\, 2)$, thus the second equality in the formula {\rm (\ref{eq2})} comes from the first one.
\end{remark}
\begin{proof}
For all integers $k\ (0\leqslant k\leqslant \lfloor\frac{n}{2}\rfloor)$, put $d(n,k)=v_{3}(\frac{n!}{(n-2k)!k!k!}2^{k}3^{n-2k})$. Then
\begin{eqnarray}\label{eqn4}
d(n,k)=n-2k+\sum\limits_{i=1}^{m}\big(\lfloor\frac{n}{3^{i}}\rfloor-\lfloor\frac{n-2k}{3^{i}}\rfloor-2\lfloor\frac{k}{3^{i}}\rfloor\big).
\end{eqnarray}
Note here that every term in the summation is nonnegative.
By {\bf Lemma \ref{lem1}},  to conclude {\bf Theorem \ref{thm2}}, we need only show the following formulas
\begin{eqnarray}
d\big(n,\lfloor\frac{n}{2}\rfloor\big)&=&(n\bmod  2)+\sum\limits_{i=1}^{r}(-1)^{r-i}s_{i},\label{eq3}\\
d(n,k)&>&(n \bmod  2)+\sum\limits_{i=1}^{r}(-1)^{r-i}s_{i}\ \big(0\leqslant k<\lfloor\frac{n}{2}\rfloor\big)\label{eq4}.
\end{eqnarray}

We begin with the proof of the formula (\ref{eq3}).

If $r=0$, then $n(j)\in \{0,2\}$ for $0\leqslant j\leqslant m$, thus $\lfloor\frac{n}{3^{i}}\rfloor=2\lfloor\frac{\lfloor\frac{n}{2}\rfloor}{3^{i}}\rfloor$ for $1\leqslant i\leqslant m$.
Hence by the formula (\ref{eqn4}), we obtain $d\big(n,\lfloor\frac{n}{2}\rfloor\big)=0$, thus the formula (\ref{eq3}) holds.

From now on, we suppose $r\geqslant 1$. Write $\lfloor\frac{n}{2}\rfloor=\sum\limits_{i=0}^me(i)3^i$,
with $e(i)\in \{0,1,2\}$, and we distinguish two cases below, according to the parity of $r$.
\vskip 5pt

{\bf Case 1:} $r$ is even. Set $s_0=-1$. By the definition of $s_j\ (1\leqslant j\leqslant r)$, we have
\begin{eqnarray*}
n&=&\sum_{i=s_r+1}^mn(i)3^i+\sum_{j=0}^{\frac{r}{2}-1}\Big(\sum_{i=s_{r-2j-1}}^{s_{r-2j}}n(i)3^i+\sum_{i=s_{r-2j-2}+1}^{s_{r-2j-1}-1}n(i)3^i\Big)\\
 &=&\sum_{i=s_r+1}^mn(i)3^i+\sum_{j=0}^{\frac{r}{2}-1}\Big(3^{s_{r-2j}}+\sum_{i=s_{r-2j-1}+1}^{s_{r-2j}-1}n(i)3^i+3^{s_{r-2j-1}}\\
 &&+\sum_{i=s_{r-2j-2}+1}^{s_{r-2j-1}-1}n(i)3^i\Big)\\
 &=&\sum_{i=s_r+1}^mn(i)3^i+\sum_{j=0}^{\frac{r}{2}-1}\Big(0\cdot 3^{s_{r-2j}}+\sum_{i=s_{r-2j-1}+1}^{s_{r-2j}-1}(2+n(i))3^i\\
 &&+4\cdot 3^{s_{r-2j-1}}+\sum_{i=s_{r-2j-2}+1}^{s_{r-2j-1}-1}n(i)3^i\Big),
\end{eqnarray*}
from which we deduce directly, for $1\leqslant i\leqslant m$,
\begin{eqnarray}
e(i)&=&\left\{\begin{array}{cl}
\frac{n(i)}{2}, &\textrm{for}\ s_{r-2j-2}< i<s_{r-2j-1}\ (0\leqslant j<\frac{r}{2}),\\
2,&\textrm{for}\ i=s_{r-2j-1}\ (0\leqslant j<\frac{r}{2}), \\
1+\frac{n(i)}{2}, &\textrm{for}\ s_{r-2j-1}< i<s_{r-2j}\ (0\leqslant j<\frac{r}{2}),\\
0,&\textrm{for}\ i=s_{r-2j}\ (0\leqslant j<\frac{r}{2}), \\
\frac{n(i)}{2}, &\textrm{for}\ i>s_r,
\end{array}
\right.\label{eqn7}\\
\alpha_i&=&\left\{\begin{array}{cl}
0, &\textrm{for}\ s_{r-2j-2}<  i\leqslant s_{r-2j-1}\ (0\leqslant j<\frac{r}{2}),\\
1, &\textrm{for}\ s_{r-2j-1}< i\leqslant s_{r-2j}\ (0\leqslant j<\frac{r}{2}),\\
0, &\textrm{for}\ i>s_r.
\end{array}
\right.\label{eqn8}
\end{eqnarray}
where $\alpha_i:=\lfloor\frac{n}{3^{i}}\rfloor-2\lfloor\frac{\lfloor n/2\rfloor}{3^{i}}\rfloor=\sum\limits_{u=i}^{m}n(u)3^{u-i}-2\sum\limits_{u=i}^{m}e(u)3^{u-i}$.

Consequently we obtain
\begin{eqnarray*}
d\big(n,\lfloor\frac{n}{2}\rfloor\big)&=&(n-2\lfloor\frac{n}{2}\rfloor)+\sum\limits_{i=1}^{m}\alpha_i\\
&=&(n \bmod 2)+\sum\limits_{j=0}^{\frac{r}{2}-1}\sum\limits_{i=s_{r-2j-1}+1}^{s_{r-2j}}1\\
&=&(n \bmod 2)+\sum\limits_{j=0}^{\frac{r}{2}-1}(s_{r-2j}-s_{r-2j-1})\\
&=&(n \bmod 2)+\sum\limits_{i=1}^{r}(-1)^{r-i}s_{i}.
\end{eqnarray*}

{\bf Case 2:} $r$ is an odd. Set $s_0=0$. As above, we have
\begin{eqnarray*}
n&=&\sum_{i=s_r+1}^mn(i)3^i+\sum_{j=0}^{\frac{r-1}{2}-1}\Big(\sum_{i=s_{r-2j-1}}^{s_{r-2j}}n(i)3^i+\sum_{i=s_{r-2j-2}+1}^{s_{r-2j-1}-1}n(i)3^i\Big)
+\sum_{i=0}^{s_{1}}n(i)3^i\\
 &=&\sum_{i=s_r+1}^mn(i)3^i+\sum_{j=0}^{\frac{r-1}{2}-1}\Big(3^{s_{r-2j}}+\sum_{i=s_{r-2j-1}+1}^{s_{r-2j}-1}n(i)3^i+3^{s_{r-2j-1}}\\
 &&+\sum_{i=s_{r-2j-2}+1}^{s_{r-2j-1}-1}n(i)3^i\Big)+\Big(n(s_1)3^{s_1}+\sum_{i=0}^{s_{1}-1}n(i)3^i\Big)\\
 &=&\sum_{i=s_r+1}^mn(i)3^i+\sum_{j=0}^{\frac{r-1}{2}-1}\Big(0\cdot 3^{s_{r-2j}}+\sum_{i=s_{r-2j-1}+1}^{s_{r-2j}-1}(2+n(i))3^i\\
 &&+4\cdot 3^{s_{r-2j-1}}+\sum_{i=s_{r-2j-2}+1}^{s_{r-2j-1}-1}n(i)3^i\Big)+\Big(3^{s_1}+\sum_{i=0}^{s_{1}-1}n(i)3^i\Big),
\end{eqnarray*}
where the last term equals $1$ if $s_1=0$, and $0\cdot 3^{s_1}+\sum\limits_{i=1}^{s_{1}-1}(2+n(i))3^i+(3+n(0))$ if $s_1>0$,
from which we obtain at once, for $1\leqslant i\leqslant m$,
\begin{eqnarray}
e(i)&=&\left\{\begin{array}{cl}
\frac{n(i)}{2}, &\textrm{for}\ s_{r-2j-2}< i<s_{r-2j-1}\ (0\leqslant j<\frac{r-1}{2}),\\
2,&\textrm{for}\ i=s_{r-2j-1}\ (0\leqslant j< \frac{r-1}{2}), \\
1+\frac{n(i)}{2}, &\textrm{for}\ s_{r-2j-1}< i<s_{r-2j}\ (0\leqslant j\leqslant \frac{r-1}{2}),\\
0,&\textrm{for}\ i=s_{r-2j}\ (0\leqslant j\leqslant \frac{r-1}{2}), \\
\frac{n(i)}{2}, &\textrm{for}\ i>s_r.
\end{array}
\right.\label{eqn9}\\
\alpha_i&=&\left\{\begin{array}{cl}
0, &\textrm{for}\ s_{r-2j-2}<  i\leqslant s_{r-2j-1}\ (0\leqslant j<\frac{r-1}{2}),\\
1, &\textrm{for}\ s_{r-2j-1}< i\leqslant s_{r-2j}\ (0\leqslant j\leqslant \frac{r-1}{2}),\\
0, &\textrm{for}\ i>s_r,
\end{array}
\right.\label{eqn10}
\end{eqnarray}
and we proceed as in the {\bf Case 1} to obtain the desired equality (\ref{eq3}).
\vskip 5pt

Now we show the formula (\ref{eq4}).

Fix $k\ (0\leqslant k<\lfloor\frac{n}{2}\rfloor)$ an integer, and write
$k=\sum\limits_{i=0}^mk(i)3^i$, with $k(i)\in \{0,1,2\}$. Let $t$ be the largest index $i\ (0\leqslant i\leqslant m)$  such that $k(i)\neq e(i)$.
Then $k(t)\leqslant e(t)-1$ and $k(i)=e(i)$ for $t<i\leqslant m$. Put $s_{-1}=s_0=0$, and $s_{r+1}=+\infty$. Then
\begin{eqnarray}\label{eqn11}
n-2k&=&(n \bmod 2)+\big(2\lfloor\frac{n}{2}\rfloor-2k\big)\\
&=&(n \bmod 2)+2\sum\limits_{i=0}^{t}\big(e(i)-k(i)\big)3^{i}\nonumber\\
&\geqslant &(n \bmod 2)+2\cdot 3^t+2\sum\limits_{i=0}^{t-1}\big(e(i)-2\big)3^{i}\nonumber\\
&=&(n \bmod 2)+2+2\sum\limits_{i=0}^{t-1}e(i)3^{i}.\nonumber
\end{eqnarray}

Now we distinguish two cases below according to the property of $t$.
\vskip 5pt

{\bf Case 1}: $s_{r-2j}\leqslant t<s_{r-2j+1}$ for some integers $j\ (0\leqslant j\leqslant \lfloor\frac{r}{2}\rfloor)$. So $\alpha_{t+1}=0$,
thus
$\sum\limits_{u=t+1}^{m}n(u)3^{u}=2\sum\limits_{u=t+1}^{m}e(u)3^{u}$,
from which we obtain, by the definition of $t$,
\begin{eqnarray*}
0\leqslant n-2k=\sum_{i=0}^mn(i)3^i-2\sum_{i=0}^mk(i)3^i=\sum_{i=0}^tn(i)3^i-2\sum_{i=0}^tk(i)3^i<3^{t+1}.
\end{eqnarray*}
Hence $\lfloor\frac{n-2k}{3^{i}}\rfloor=0$ for all integers $i>t$, and by the formulas (\ref{eqn8}) and (\ref{eqn10}),
\begin{eqnarray}\label{eqn12}
&&\sum\limits_{i=t+1}^{m}\Big(\lfloor\frac{n}{3^{i}}\rfloor-\lfloor\frac{n-2k}{3^{i}}\rfloor
-2\lfloor\frac{k}{3^{i}}\rfloor\Big)
=\sum\limits_{i=t+1}^{m}\alpha_i\geqslant \sum\limits_{i=s_{r-2j+1}}^{m}\alpha_i\\
&\geqslant &\sum_{u=0}^{j-1}\sum_{i=s_{r-2u-1}+1}^{s_{r-2u}}\alpha_i
= \sum\limits_{u=0}^{j-1}(s_{r-2u}-s_{r-2u-1}).\nonumber
\end{eqnarray}
By the formulas (\ref{eqn7}) and (\ref{eqn9}), we know that for $s_{r-2u-1}\leqslant i< s_{r-2u}\ (0\leqslant u\leqslant \lfloor\frac{r}{2}\rfloor)$, we have $e(i)\geqslant 1$.
Consequently we obtain
\begin{eqnarray}\label{eqn13}
\sum\limits_{i=0}^{t-1}e(i)\geqslant \sum_{u=j}^{\lfloor\frac{r}{2}\rfloor}\sum\limits_{i=s_{r-2u-1}}^{s_{r-2u}-1}e(i)
\geqslant  \sum_{u=j}^{\lfloor\frac{r}{2}\rfloor}(s_{r-2u}-s_{r-2u-1}).
\end{eqnarray}
By Combining the formulas (\ref{eqn11}), (\ref{eqn13}), and (\ref{eqn12}), we obtain
\begin{eqnarray*}
d(n,k)&=&(n-2k)+\sum\limits_{i=1}^{m}\Big(\lfloor\frac{n}{3^{i}}\rfloor-\lfloor\frac{n-2k}{3^{i}}\rfloor-2\lfloor\frac{k}{3^{i}}\rfloor\Big)\\
&>&(n \bmod 2)+\sum_{u=j}^{\lfloor\frac{r}{2}\rfloor}(s_{r-2u}-s_{r-2u-1})+
\sum\limits_{u=0}^{j-1}(s_{r-2u}-s_{r-2u-1})\\
&=&(n \bmod 2)+\sum\limits_{i=1}^{r}(-1)^{r-i}s_{i}.
\end{eqnarray*}

{\bf Case 2}:  $s_{r-2j-1}\leqslant t< s_{r-2j}$ for some integers $j\ (0\leqslant j\leqslant \lfloor\frac{r}{2}\rfloor)$.
Then $\alpha_{t+1}=1$, thus
$\sum\limits_{u=t+1}^{m}n(u)3^{u}=3^{t+1}+2\sum\limits_{u=t+1}^{m}e(u)3^{u}$,
and by the definition of $t$, we obtain
\begin{eqnarray*}
0\leqslant n-2k=\sum_{i=0}^mn(u)3^u-2\sum_{i=0}^mk(i)3^i=3^{t+1}\sum_{i=0}^tn(u)3^u-2\sum_{i=0}^tk(i)3^i<2\cdot 3^{t+1}.
\end{eqnarray*}
Hence $0\leqslant \lfloor\frac{n-2k}{3^{t+1}}\rfloor\leqslant 1$, and $\lfloor\frac{n-2k}{3^{i}}\rfloor=0$ for all integers $i>t+1$,
from which, we deduce, by the formulas (\ref{eqn8}) and (\ref{eqn10}),
\begin{eqnarray}\label{eqn14}
&&\sum\limits_{i=t+1}^{m}\Big(\lfloor\frac{n}{3^{i}}\rfloor-\lfloor\frac{n-2k}{3^{i}}\rfloor
-2\lfloor\frac{k}{3^{i}}\rfloor\Big)
\geqslant -1+\sum\limits_{i=t+1}^{m}\alpha_i\\
&\geqslant &-1+\sum\limits_{i=t+1}^{s_{r-2j}}\alpha_i+\sum_{u=0}^{j-1}\sum_{i=s_{r-2u-1}+1}^{s_{r-2u}}\alpha_i\nonumber\\
&=&-1+(s_{r-2j}-t)+\sum\limits_{u=0}^{j-1}(s_{r-2u}-s_{r-2u-1}).\nonumber
\end{eqnarray}
By the formulas (\ref{eqn7}) and (\ref{eqn9}), we know that for $s_{r-2u-1}\leqslant i < s_{r-2u}\ (0\leqslant u\leqslant \lfloor\frac{r}{2}\rfloor)$, we have $e(i)\geqslant 1$.
Consequently we obtain
\begin{eqnarray}\label{eqn15}
\sum\limits_{i=0}^{t-1}e(i)&\geqslant&\sum\limits_{i=s_{r-2j-1}}^{t-1}e(i)+ \sum_{u=j+1}^{\lfloor\frac{r}{2}\rfloor}\sum\limits_{i=s_{r-2u-1}}^{s_{r-2u}-1}e(i)\\
&\geqslant &(t-s_{r-2j-1})+  \sum_{u=j+1}^{\lfloor\frac{r}{2}\rfloor}(s_{r-2u}-s_{r-2u-1}).\nonumber
\end{eqnarray}
By Combining the formulas (\ref{eqn11}), (\ref{eqn15}), and (\ref{eqn14}), we obtain
\begin{eqnarray*}
d(n,k)&=&(n-2k)+\sum\limits_{i=1}^{m}\Big(\lfloor\frac{n}{3^{i}}\rfloor-\lfloor\frac{n-2k}{3^{i}}\rfloor-2\lfloor\frac{k}{3^{i}}\rfloor\Big)\\
&\geqslant &(n \bmod 2)+2+(t-s_{r-2j-1})+\sum_{u=j+1}^{\lfloor\frac{r}{2}\rfloor}(s_{r-2u}-s_{r-2u-1})\\
&&-1+(s_{r-2j}-t)+\sum\limits_{u=0}^{j-1}(s_{r-2u}-s_{r-2u-1})\\
&>&(n \bmod 2)+\sum\limits_{i=1}^{r}(-1)^{r-i}s_{i}.
\end{eqnarray*}
Hence the formula (\ref{eq6}) holds, and this finished the proof of {\bf Theorem \ref{thm2}}.
\end{proof}

As application of {\bf Theorem \ref{thm2}}, we obtain at once the following result.

\begin{theorem}\label{thm3}
For all integers $n\geqslant 0$, we have
\begin{equation}
b(n)=\left\{\begin{array}{lcl}
&b\big(\lfloor\frac{n}{3}\rfloor\big)+ \big(\lfloor\frac{n}{3}\rfloor \bmod 2\big), &\text{if } n\equiv 0,2 \ (\bmod\, 3),\\
&b\big(\lfloor\frac{n}{3}\rfloor\big)+1-\big(\lfloor\frac{n}{3}\rfloor \bmod 2\big), &\text{if } n\equiv 1 \ (\bmod\, 3).
\end{array}\right.
\end{equation}
\end{theorem}
\begin{proof} Fix $n\geqslant 1$ an integer, and write $n=\sum\limits_{j=0}^mn(j)3^j$, with $m=\lfloor \frac{\log n}{\log 3}\rfloor$, and  $n(j)\in \{0,1,2\}\ (0\leqslant j\leqslant m)$.
Let $s_{1}<s_{2}<\cdots <s_{r}$ be all the indices $j$ (if there exist) such that $n(j)=1$. We put $r=0$ if $n(j)\neq 1$ for all integers $j\ (0\leqslant j\leqslant m)$.

If $r=0$, then $n(j)\equiv 0\ (\bmod 2)$ for $0\leqslant j\leqslant m$, hence $n\equiv \lfloor \frac{n}{3}\rfloor \equiv 0\ (\bmod 2)$, and by {\bf Theorem \ref{thm2}},
we obtain $b(n)=0=b\big(\lfloor \frac{n}{3}\rfloor\big)+\big(\lfloor \frac{n}{3}\rfloor \bmod 2\big)$.

From now on, we suppose $r\geqslant 1$, and distinguish two cases below.

{\bf Case 1:} $n\equiv 0,2 \ (\bmod\, 3)$. Then $s_{1}>0$, and $n\equiv \lfloor\frac{n}{3}\rfloor\ (\bmod\, 2)$. Hence by {\bf Theorem \ref{thm2}}, we obtain
\begin{eqnarray*}
b(n)&=&(n\bmod 2)+\sum\limits_{i=1}^{r}(-1)^{r-i}s_{i}=\big(\lfloor\frac{n}{3}\rfloor\bmod 2\big)+\sum\limits_{i=1}^{r}(-1)^{r-i}s_{i},\\
b\big(\lfloor\frac{n}{3}\rfloor\big)&=&\sum\limits_{i=1}^{r}(-1)^{r-i}\big((s_{i}-1)+1\big)=\sum\limits_{i=1}^{r}(-1)^{r-i}s_{i},
\end{eqnarray*}
from which we deduce at once $b(n)=b\big(\lfloor\frac{n}{3}\rfloor\big)+\big(\lfloor\frac{n}{3}\rfloor\bmod 2\big)$.

{\bf Case 2:} $ n\equiv 1 \ (\bmod\, 3)$. Then $s_{1}=0$, and $n\equiv 1+\lfloor\frac{n}{3}\rfloor\ (\bmod\, 2)$. Hence by {\bf Theorem \ref{thm2}}, we have
\begin{eqnarray*}
b(n)&=&(n\bmod 2)+\sum\limits_{i=1}^{r}(-1)^{r-i}s_{i}=1-\big(\lfloor\frac{n}{3}\rfloor\bmod 2\big)+\sum\limits_{i=1}^{r}(-1)^{r-i}s_{i},\\
b\big(\lfloor\frac{n}{3}\rfloor\big)&=&\sum\limits_{i=2}^{r}(-1)^{r-i}((s_{i}-1)+1)=\sum\limits_{i=1}^{r}(-1)^{r-i}s_{i}.
\end{eqnarray*}
from which we obtain $b(n)=b\big(\lfloor\frac{n}{3}\rfloor\big)+1-\big(\lfloor\frac{n}{3}\rfloor\bmod 2\big)$.
\end{proof}

\subsection{Proof of Theorem \ref{thm1}}
Now we are ready to establish {\bf Theorem \ref{thm1}}. By {\bf Theorem~\ref{thm3}}, we need only consider the case that $n\equiv 1 \ (\bmod\, 3)$. Write $n=3k+1$, with $k\geqslant 0$ an integer.
We distinguish two cases below.

{\bf Case 1:} $k\equiv 0,2 \ (\bmod\, 3)$. Then by {\bf Theorem \ref{thm3}}, we have
\begin{eqnarray*}
b(n)&=&b(k)+1-(k \bmod 2\big)\\
    &=&b\big(\lfloor\frac{k}{3}\rfloor\big)+ \big(\lfloor\frac{k}{3}\rfloor \bmod 2\big)+1-(k \bmod 2\big)\\
    &=&b\big(\lfloor\frac{n}{9}\rfloor\big)+1,
\end{eqnarray*}
for we have $k\equiv 3\lfloor\frac{k}{3}\rfloor \equiv \lfloor\frac{k}{3}\rfloor \, (\bmod 2)$.

{\bf Case 2:} $k\equiv 1 \ (\bmod\, 3)$. Then by {\bf Theorem \ref{thm3}}, we have
\begin{eqnarray*}
b(n)&=&b(k)+1-(k \bmod 2\big)\\
    &=&b\big(\lfloor\frac{k}{3}\rfloor\big)+1- \big(\lfloor\frac{k}{3}\rfloor \bmod 2\big)+1-(k \bmod 2\big)\\
    &=&b\big(\lfloor\frac{n}{9}\rfloor\big)+1,
\end{eqnarray*}
for we have $k=3\lfloor\frac{k}{3}\rfloor +1\equiv \lfloor\frac{k}{3}\rfloor+1\, (\bmod\, 2)$.

\subsection{$p$-regular sequences}

Let $p\geqslant 2$ be an integer. According to J.-P. Allouche and J. Shallit \cite{as1} (see also \cite{as3}), a sequence $(u(n))_{n\geqslant 0}$ with values in $\mathbb{Z}$ is $p$-regular if the $\mathbb{Z}$-module generated by the following set
$$
K_p(u)=\Big\{\big(u(p^bn+a)\big)_{n\geqslant 0}\ |\ b\geqslant 0, 0\leqslant a<p^b\Big\}
$$
is a finitely generated $\mathbb{Z}$-module.

As corollary of {\bf Theorem \ref{thm3}}, we have the following result.

\begin{theorem}\label{thm4}
The sequence $(b(n))_{n\geqslant 0}$ is $3$-regular.
\end{theorem}
\begin{proof} For all integers $n\geqslant 0$, define
\begin{eqnarray*}
V(n)=\bigg(\begin{array}{ccc}b(n)\\
1\\
n\ \bmod 2
\end{array}\bigg).
\end{eqnarray*}
By {\bf Theorem \ref{thm3}}, we have $V(3n+k)=\mu(k)V(n)$, for $0\leqslant k\leqslant 2$, where
\begin{eqnarray*}
\mu(0)=\mu(2)=\bigg(\begin{array}{ccc} 1 &0 &1\\
0 &1 &0\\
0 &0 &1
\end{array}\bigg), \ \textrm{and}\
\mu(1)=\bigg(\begin{array}{ccc} 1 &1 &-1\\
0 &1 &0\\
0 &1 &-1
\end{array}\bigg).
\end{eqnarray*}
Thus the sequence $(b(n))_{n\geqslant 0}$ is $3$-regular, by {\bf Theorem 16.1.3} in \cite[p.~439]{as3}.
\end{proof}
\vskip 5pt

\textbf{Acknowledgments.} The author would like to thank heartily Professor Jia-Yan YAO
for interesting discussions and for his help in the preparation of the work. He would like also to thank warmly
Professor Jeffrey Shallit for informing him the background of the problem in discussion.
Finally he would like to thank the National Natural Science Foundation of China (Grants No. 11871295) for partial financial support.

\end{document}